[1994/12/01]
\documentclass[manuscript, a4paper, reqno]{aomart}

\chardef\bslash=`\\ 

\newtheorem[{}\it]{thm}{Theorem}[section]
\newtheorem{cor}[thm]{Corollary}
\newtheorem{lem}[thm]{Lemma}

\theoremstyle{remark} 

\usepackage{dsfont}
\usepackage{amssymb}
\usepackage[misc]{ifsym}

\usepackage{graphicx}
\usepackage{moresize}

\usepackage{hyperref}
\usepackage[numeric]{amsrefs}

\usepackage[english]{babel}

\theoremstyle{definition}
\newtheorem{defn}{\textsc{Definition}}[section]
\newtheorem{rem}{Remark}[section]

\newtheorem*[{}\it]{notation}{Notation}

\newtheorem*[{}\it]{rest}{\textsc{Theorem}}
\newtheorem*[{}\it]{quest}{\textsc{Question}}
\newtheorem*[{}\it]{problemo}{\textsc{Problem}}
\newtheorem*[{}\it]{projone}{\texttt{Project 1 (Future Work)}}
\newtheorem*[{}\it]{projtwo}{\texttt{Project 2 (Current Work)}}
\newtheorem*[{}\it]{projthree}{\texttt{Project 3 (Current Work)}}
\newtheorem*[{}\it]{projfour}{\texttt{Project 4 (Future Work)}}
\newtheorem*[{}\it]{projfive}{\texttt{Project 5 (Future Work)}}
\newtheorem*[{}\it]{proofoflemma}{Proof of Lemma}

\usepackage[textwidth=6in,textheight=9.25in, paperwidth=7.5in,paperheight=11.25in, inner=2.5cm, outer=2.5cm]{geometry} 

\usepackage[usenames,dvipsnames]{xcolor}

\usepackage{txfonts}

\usepackage{mathrsfs}

\usepackage{multicol}

\usepackage{url}

\usepackage{enumerate}

\title[]{On the Initial Boundary-Value Problem \\ in the Kinetic Theory of Hard Particles I: Non-existence}
\author[]{Mark Wilkinson}\thanks{Department of Mathematics, Heriot-Watt University and the Maxwell Institute for Mathematical Sciences, Edinburgh, Scotland (\Letter) \href{mailto:mark.wilkinson@hw.ac.uk}{mark.wilkinson@hw.ac.uk}}

\allowdisplaybreaks

\newcommand{\ov}{\overline}

\usepackage{tikz}
\usetikzlibrary{patterns}
\newcommand{\boundellipse}[3]
{(#1) ellipse (#2 and #3)
}

\hypersetup{citecolor=red}

\begin{document}

\maketitle

\begin{abstract}
{\noindent In the first of two papers, we study the initial boundary-value problem that underlies the theory of the Boltzmann equation for general non-spherical hard particles. In this work, for two congruent ellipses and for a large class of associated boundary conditions, we identify initial conditions for which there do not exist local-in-time weak solutions of Newton's equations of motion. To our knowledge, this is the first time the necessity of {\em rolling} in the energy-conserving dynamics of strictly-convex rigid bodies has been demonstrated. This study was, in part, motivated by a recent observation of \textsc{Palffy-Muhoray, Virga, Wilkinson and Zheng} \cite{palffy2017paradox} on the interpenetration of strictly-convex rigid bodies. 
}
\end{abstract}
%

\vspace{3mm}
\section{Introduction}
Let $d=2, 3$. Suppose $\mathsf{P}_{\ast}$ is a compact, connected subset of $\mathbb{R}^{d}$ whose boundary is rectifiable. Consider an evolution of two congruent copies of this set in $\mathbb{R}^{d}$ expressed by
\begin{equation}\label{setevo}
\mathsf{P}(t):=R(t)\mathsf{P}_{\ast}+x(t) \qquad \textrm{and} \qquad \ov{\mathsf{P}}(t):=\ov{R}(t)\mathsf{P}_{\ast}+\ov{x}(t),
\end{equation}
where the associated dynamics 
\begin{equation*}
t\mapsto X(t):=[x(t), \ov{x}(t), R(t), \ov{R}(t)]\in \mathcal{M}_{d}:=\mathbb{R}^{d}\times\mathbb{R}^{d}\times\mathrm{SO}(d)\times\mathrm{SO}(d)
\end{equation*} 
on $\mathbb{R}$ takes its range in the set
\begin{equation*}
\mathcal{P}_{d}(\mathsf{P}_{\ast}):=\left\{X\in\mathcal{M}_{d}\,:\,\mathscr{L}_{d}((R\mathsf{P}_{\ast}+x)\cap(\ov{R}\mathsf{P}_{\ast}+\ov{x}))=0\right\},
\end{equation*}
where $\mathscr{L}_{d}$ denotes the $d$-dimensional Lebesgue measure on $\mathbb{R}^{d}$. In this article, we address the problem of the assignment of boundary conditions on the boundary of the set of physically-admissible configurations
\begin{equation*}
\partial\mathcal{P}_{d}(\mathsf{P}_{\ast}):=\left\{X\in\mathcal{M}_{d}\,:\,
\begin{array}{c}
(R\mathsf{P}_{\ast}+x)\cap(\ov{R}\mathsf{P}_{\ast}+\ov{x})\neq \varnothing,\\ 
\mathrm{and}\hspace{2mm}\mathscr{L}_{d}((R\mathsf{P}_{\ast}+x)\cap(\ov{R}\mathsf{P}_{\ast}+\ov{x}))=0
\end{array}
\right\}
\end{equation*}
in the construction of any global-in-time dynamics describing the evolution of the sets. Indeed, henceforth we call any $X\in\partial\mathcal{P}_{d}(\mathsf{P}_{\ast})$ a {\bf collision configuration} of the sets $\mathsf{P}$ and $\ov{\mathsf{P}}$. More specifically, in what follows we initiate a study as to how (i) the regularity of the set dynamics $t\mapsto X(t)$, and also (ii) the geometric properties of the set $\mathsf{P}_{\ast}$ affect what constitute {\bf pre-collisional} and {\bf post-collisional} velocities for a given collision configuration $X\in\partial\mathcal{P}_{d}(\mathsf{P}_{\ast})$. Understanding how to partition the set of all velocities into subsets of pre- and post-collisional velocities is an essential ingredient in the construction of {\em scattering maps} associated to the collisions of $\mathsf{P}$ and $\ov{\mathsf{P}}$ and, in turn, scattering maps are an important ingredient in the construction of global-in-time solutions of any `physical' ODE which describe their evolution. 
\subsection{A Decision Problem}\label{decision}
Suppose $N\gg1$ is an integer, and that $\mathsf{P}_{1}, ..., \mathsf{P}_{N}\subset\mathbb{R}^{d}$ are initial sets with the property $\mathscr{L}_{d}(\mathsf{P}_{i}\cap \mathsf{P}_{j})=0$ for $i\neq j$, with each $\mathsf{P}_{k}$ congruent to some fixed compact, strictly-convex $\mathsf{P}_{\ast}\subset\mathbb{R}^{d}$. In the classical kinetic theory of dilute gases, at the microscopic level one is faced with the construction of solutions -- in an appropriate sense -- of Newton's equations of motion for the time evolution of the $N$ sets $\mathsf{P}_{1}(t), ..., \mathsf{P}_{N}(t)$ given by
\begin{equation*}
\mathsf{P}_{i}(t):=R_{i}(t)\mathsf{P}_{i}+x_{i}(t)\quad \text{for}\hspace{2mm}i=1, ..., N,
\end{equation*}
where the ODE governing the phase maps are themselves given by
\begin{equation}\label{masta}
\left\{
\begin{array}{ll}
\displaystyle \frac{dx_{i}}{dt}=v_{i} & \displaystyle \frac{dR_{i}}{dt}=R_{i}\Omega_{i}, \vspace{2mm}\\
\displaystyle \frac{dv_{i}}{dt}=0, & \displaystyle\frac{d\omega_{i}}{dt}=0,
\end{array}
\right.
\end{equation}
with $\Omega_{i}\in\mathbb{R}^{3\times 3}$ being the angular velocity tensor associated to the angular velocity vector $\omega_{i}\in\mathbb{R}^{3}$. Our interest in the existence, uniqueness and qualitative properties of solutions of the initial boundary-value problem associated to \eqref{masta} stems primarily from their connection with weak solutions\footnote{There have been many (oftentimes, dimension-dependent) notions of weak solution established for the Boltzmann equation, among which lie distributional solutions of \textsc{Cercignani} \cite{Cercignani2005}, renormalised solutions of \textsc{DiPerna and Lions} \cite{diperna1989cauchy}, and mild solutions of \textsc{Ars\'{e}nio} \cite{arsenio2011global}.} of the Boltzmann equation. This important connection is provided by the so-called {\em Boltzmann-Grad limit} of the $N$-particle system as the number of gas particles becomes unbounded, namely when
\begin{equation}
\quad N(\mathrm{diam}\,\mathsf{P}_{\ast})^{d-1}=\mathcal{O}(1)\quad \text{as}\hspace{2mm}N\rightarrow\infty.
\end{equation}
More precisely, one looks to prove that a rescaled family of weak solutions of the BBGKY hierarchy associated to \eqref{masta} is precompact in a suitable function space topology, and that all limit points of this family are chaotic weak solutions of the Boltzmann hierarchy on $\mathbb{R}^{d}$. We refer the reader to the text of \textsc{Cercignani, Ilner and Pulvirenti} (\cite{cercignani2013mathematical}, chapter 2) for a discourse on the formal procedure that links solutions of \eqref{masta} with solutions of the Boltzmann equation. 

There still remain challenges in establishing a comprehensive Cauchy theory for a suitable notion of weak solution to Newton's equations for set evolution in the case when each of the initial sets $\mathsf{P}_{j}$ is non-spherical. The rigorous theory of the Boltzmann equation for non-spherical particles was recently initiated by \textsc{Saint-Raymond and Wilkinson} \cite{saint2015collision}. In that article, it was tacitly assumed that (i) there was only one `physical' family of scattering maps that resolved the collision between two non-spherical sets that conserved total linear momentum, angular moment and kinetic energy, and (ii) that it was always possible to construct an associated global-in-time weak solution for any given admissible datum in phase space. In a companion article to this work, namely \textsc{Wilkinson} \cite{me13}, we show that there are {\em infinitely-many} families of scattering maps which resolve the collision between two non-spherical sets that respect all the conservation laws of classical mechanics. In turn, this observation gives rise to the {\bf non-uniqueness} of global-in-time weak solutions to the initial boundary-value problem associated to Newton's equations on $\mathcal{P}_{d}(\mathsf{P}_{\ast})$. However, in the present article, we are able to demonstrate the {\bf non-existence} of local-in-time weak solutions of the initial boundary-value problem associated to \eqref{masta} for:
\begin{itemize}
\item the particular case where $\mathsf{P}_{\ast}$ is an ellipse $\mathsf{E}_{\ast}$ in the plane ($d=2$);
\item a large class of boundary conditions (so-called frictionless scattering maps) associated thereto, including the oft-employed {\em Boltzmann scattering}; and
\item a set of initial conditions lying in a co-dimension 1 subset of $\mathcal{P}_{2}(\mathsf{E}_{\ast})\times \mathbb{R}^{6}$.
\end{itemize}
Our main result focuses on those particles which are ellipses and ellipsoids due to the fact they are realised as the level sets of explicit algebraic functions on $\mathbb{R}^{2}$ and $\mathbb{R}^{3}$, respectively: our method of proof makes crucial use of this fact. We nevertheless expect our results to hold true for all sufficiently-smooth and strictly convex non-spherical particles. The precise notion of weak solution to system \eqref{masta} we employ in our work is given in section \ref{scatteringmaps} below. 
\subsection*{A Velocity Decision Problem}
During the na\"{i}ve construction of any dynamics $t\mapsto X(t)$ that constitutes a solution of the initial boundary-value problem associated to \eqref{masta}, one is faced with the following {\em velocity decision problem}. Suppose a collision configuration $X\in\partial\mathcal{P}_{d}(\mathsf{P}_{\ast})$ is given and fixed. Consider assigning (i) the centres of mass $x, \ov{x}\in\mathbb{R}^{d}$ each with an initial linear velocity $v, \ov{v}\in\mathbb{R}^{d}$, and (ii) the sets with an initial angular velocity $\omega, \ov{\omega}\in\mathbb{R}^{n(d)}$, where $n(2)=1$ and $n(3)=3$. For brevity, let us concatenate this initial velocity data in a single velocity vector $V:=[v, \ov{v}, \omega, \ov{\omega}]\in\mathbb{R}^{2d}\times\mathbb{R}^{2n(d)}$. We ask the following questions:
\vspace{2mm}

\begin{enumerate}[{\bf (Q1)}]
\item For the given collision configuration $X\in\partial\mathcal{P}_{d}(\mathsf{P}_{\ast})$, is $V\in\mathbb{R}^{2d}\times\mathbb{R}^{2m(d)}$ pre-collisional or post-collisional?
\item For the given collision configuration $X\in\partial\mathcal{P}_{d}(\mathsf{P}_{\ast})$, do there exist inadmissible velocity vectors $V\in\mathbb{R}^{2d}\times\mathbb{R}^{2m(d)}$, in the sense that $V$ is neither pre-collisional nor post-collisional?
\end{enumerate}
%
%
%
%
In the case of hard spheres in $\mathbb{R}^{3}$ -- or, indeed, hard disks in $\mathbb{R}^{2}$ -- the answers to both questions are well known. For instance, when $d=3$ and one takes $\mathsf{P}_{\ast}$ to be the closed ball of radius $\frac{1}{2}$ with centre at the origin in $\mathbb{R}^{3}$, the set of all velocity vectors $\Sigma_{X}^{-}\subset\mathbb{R}^{6}$ which are pre-collisional with respect to a collision configuration
\begin{equation*}
X=\left[
\begin{array}{c}
x_{0} \\
x_{0}+n
\end{array}
\right]\in\partial\mathcal{P}_{3}(\mathsf{P}_{\ast})
\end{equation*}
for some $x_{0}\in\mathbb{R}^{3}$ and $n\in\mathbb{S}^{2}$ is the half space
\begin{equation*}
\Sigma_{X}^{-}:=\left\{
V\in\mathbb{R}^{6}\,:\,\widehat{\nu}_{n}\cdot V \leq 0
\right\},
\end{equation*}
while the analogous set of post-collisional velocity vectors $\Sigma_{X}^{+}$ is
\begin{equation*}
\Sigma_{X}^{+}:=\left\{
V\in\mathbb{R}^{6}\,:\,\widehat{\nu}_{n}\cdot V \geq 0
\right\},
\end{equation*}
where
\begin{equation*}
\widehat{\nu}_{n}:=\frac{1}{\sqrt{2}}\left[
\begin{array}{c}
n \\
-n
\end{array}
\right].
\end{equation*}
In particular, there are no {\em inadmissible} velocity vectors in $\mathbb{R}^{6}$ for two hard spheres. In order to resolve this velocity decision problem in the case of more general sets, one must define carefully what one means by `pre-collisional' and `post-collisional'. We claim that to answer the problem in a meaningful way, one must specify the regularity class of the dynamics $t\mapsto X(t)$ so as to state precisely in which sense its associated velocity map $t\mapsto V(t)$ exists. In doing so, we find a way of generalising the notion of pre- and post-collisional velocities from vectors in $\mathbb{R}^{2d}\times\mathbb{R}^{2m(d)}$ to {\em velocity germs}.
\begin{rem}
In what follows, we shall see it is the choice of regularity class for the dynamics that is our key to answering {\bf (Q1)}. However, it is rather the study of the {\em geometry} of the set $\mathsf{P}_{\ast}$ which is the key to answering {\bf (Q2)}, particularly in the affirmative. We also address the (admittedly, much harder) question on how the geometry of the sets affects the minimal possible regularity class for the dynamics $t\mapsto X(t)$.
\end{rem}
\subsection{A Recent Observation on `Natural' Scattering Operators}
This investigation was prompted by a recent result of \textsc{Palffy-Muhoray, Virga, Wilkinson and Zheng}. Among other results, in \cite{palffy2017paradox} it was shown that when $\mathsf{P}_{\ast}$ is taken to be an ellipsoid, {\bf (Q2)} above can be answered in the affirmative for $\mathcal{P}_{3}(\mathsf{P}_{\ast})$ when the boundary conditions on $\partial\mathcal{P}_{3}(\mathsf{P}_{\ast})$ are taken to be those furnished by the classical Boltzmann scattering $S:\mathcal{P}_{3}(\mathsf{P}_{\ast})\times\mathbb{R}^{12}\rightarrow\mathbb{R}^{12}$, namely
\begin{equation*}
S(X, V):=[v_{X}', \ov{v}_{X}', \omega_{X}', \ov{\omega}_{X}'],
\end{equation*}
in which case the `post-collisional' linear velocities given by
\begin{equation*}
\begin{array}{c}
v_{X}':=v-m^{-1}\Lambda_{X}^{-1}[(v_{P}-v_{Q})\cdot n_{X}]n_{X}, \vspace{1mm}\\
\ov{v}_{X}':=\ov{v}+m^{-1}\Lambda_{X}^{-1}[(v_{P}-v_{Q})\cdot n_{X}]n_{X},
\end{array}
\end{equation*}
and the `post-collisional' angular velocities given by
\begin{equation*}
\begin{array}{c}
\omega_{X}':=\omega-\Lambda_{X}^{-1}[(v_{P}-v_{Q})\cdot n_{X}]J(p_{X}\wedge n_{X}),\vspace{1mm}\\
\ov{\omega}_{X}'=\ov{\omega}+\Lambda_{X}^{-1}[(v_{P}-v_{Q})\cdot n_{X}]J(q_{X}\wedge n_{X}),
\end{array}
\end{equation*}
where the particle mass is 
\begin{equation*}
m:=\int_{\mathsf{P}_{\ast}}\,dy
\end{equation*}
and the inertia tensor is
\begin{equation*}
J:=\int_{\mathsf{P}_{\ast}}\left(|y|^{2}I-y\otimes y\right)\,dy.
\end{equation*}
In addition,
\begin{equation*}
v_{P}:=v+\omega\wedge p_{X}, \quad v_{Q}:=\ov{v}+\ov{\omega}\wedge q_{X},
\end{equation*}
and
\begin{equation*}
\Lambda_{X}:=\frac{1}{2}\left(\frac{2}{m}+\left|\sqrt{J}p_{X}\wedge n_{X}\right|^{2}+\left|\sqrt{J}q_{X}\wedge n_{X}\right|^{2}\right),
\end{equation*}
while the spatial quantities $p_{X}, q_{X}\in\mathbb{R}^{3}$ and $n_{X}\in\mathbb{S}^{2}$ in the above are defined in section \ref{parz} below. The authors show there exist collision configurations $X\in\partial\mathcal{P}_{3}(\mathsf{P}_{\ast})$ and velocity vectors in the associated set of so-called grazing velocities 
\begin{equation}
\Sigma_{X}^{0}:=\left\{
V\in\mathbb{R}^{12}\,:\,S(X, V)=V
\right\}
\end{equation}
for which $\mathscr{L}_{3}(\mathsf{P}(t)\cap\ov{\mathsf{P}}(t))=0$ for $t\leq \tau$, but $\mathscr{L}_{3}(\mathsf{P}(t)\cap\ov{\mathsf{P}}(t))>0$ for all $t>\tau$ in a sufficiently small right neighbourhood of $\tau$, where $\tau\in\mathbb{R}$ is any collision time. They also locate collision configurations which give rise to interpenetration for {\em both} $t<\tau$ {\em and} $t>\tau$. As such, the boundary conditions for the initial boundary-value problem furnished by the classical Boltzmann scattering $S$ can lead to interpenetration of the non-spherical sets $\mathsf{P}(t)$ and $\ov{\mathsf{P}}(t)$. The results in this work offer a rigorous elucidation of these observations.
\subsection{Statement of Main Result}
For notational simplicity only, we shall work with the case $d=2$ in all that follows. As mentioned above, in this article we adopt the viewpoint that pre- and post-collisional velocities (with respect to a given collision configuration $X\in \partial\mathcal{P}_{2}(\mathsf{P}_{\ast})$) are not best viewed as subsets of $\mathbb{R}^{6}$, but rather as subsets $\Sigma_{X}^{\Box}(\mathsf{P}_{\ast}; \mathcal{X})$ of a quotient vector space $G(\mathsf{P}_{\ast}; \mathcal{X})$ of germs of $L^{1}_{\mathrm{loc}}(\mathbb{R}, \mathbb{R}^{6})$ maps, where $\mathcal{X}$ is a subspace of $L^{1}_{\mathrm{loc}}(\mathbb{R}, \mathbb{R}^{6})$ that fixes the regularity of the dynamics $t\mapsto X(t)$, and $\Box$ is a label that assumes the values
\begin{equation}
\Box=\left\{
\begin{array}{ll}
- & \quad \text{for pre-collisional velocity germs}, \vspace{2mm}\\
+ & \quad \text{for post-collisional velocity germs}, \vspace{2mm}\\
0 & \quad \text{for grazing velocity germs}, \vspace{2mm}\\
\times & \quad \text{for inadmissible velocity germs}. \vspace{2mm}\\
\end{array}
\right.
\end{equation}
The definition of these sets appears in section \ref{smoothy} below. We remark in passing that when the dynamics under study is sufficiently regular, the sets $\Sigma_{X}^{\Box}(\mathsf{P}_{\ast}; \mathcal{X})$ can be shown to be bijectively equivalent to subsets of $\mathbb{R}^{6}$. Let us now specify the class of sets we consider that model the dynamics of gas particles.
\begin{defn}
Let $\mathcal{C}$ denote the class of all compact, connected subsets of $\mathbb{R}^{3}$ whose boundaries admit the structure of a real-analytic manifold. We endow $\mathcal{C}$ with the topology generated by the Hausdorff distance on $2^{\mathbb{R}^{d}}\times2^{\mathbb{R}^{3}}$.
\end{defn}
The following is the main result of the paper.
\begin{thm}\label{mainresult}
The class of inadmissible velocities $\mathcal{G}^{\times}(\mathsf{P}_{\ast}; \mathcal{F})$ is unstable with respect to convergence in $\mathcal{C}$, in the sense that there exist convergent sequences of sets $\{\mathsf{P}_{j}\}_{j=1}^{\infty}$ such that
\begin{equation*}
\mathcal{G}^{\times}(\mathsf{P}_{j}; \mathcal{F})\neq\{\varnothing\},
\end{equation*}
but
\begin{equation*}
\mathcal{G}^{\times}(\mathsf{P}_{\ast}; \mathcal{F})=\{\varnothing\},
\end{equation*}
where $\mathsf{P}_{j}\rightarrow \mathsf{P}_{\ast}$ in $\mathcal{C}$ as $j\rightarrow\infty$, and $\mathcal{F}=\mathcal{F}(\mathbb{R}, \mathbb{R}^{6})$ is the subspace of $L^{1}_{\mathrm{loc}}(\mathbb{R}, \mathbb{R}^{6})$ given by
\begin{equation}
\mathcal{F}(\mathbb{R}, \mathbb{R}^{6}):=\left\{W\in L^{1}_{\mathrm{loc}}(\mathbb{R}, \mathbb{R}^{6})\,:\, W=\mathrm{const.} \hspace{2mm}\text{a.e. on}\hspace{2mm}\mathbb{R}\right\}
\end{equation}
\end{thm}
Informally, our result essentially says that `infinitesimal perturbations' in $\mathcal{C}$ of the shape of convex sets can dramatically affect what constitute their pre- and post-collisional velocities. The significance of establishing our main result for the subspace $\mathcal{F}(\mathbb{R}, \mathbb{R}^{6})\subset L^{1}_{\mathrm{loc}}(\mathbb{R}, \mathbb{R}^{6})$ is that the associated dynamics $t\mapsto X(t)$ is piecewise linear and thereby constitutes a `natural' class in which to seek weak solutions of the initial boundary-value problem associated to \eqref{masta} for the collisions of hard particles. Indeed, the following is an immediate, and perhaps more readily-appreciable, consequence of theorem \ref{mainresult}:
\begin{cor}
Suppose $\mathsf{P}_{\ast}$ is congruent to an ellipse in $\mathbb{R}^{2}$. Consider the initial boundary-value problem for \eqref{masta} in which the boundary conditions are realised by the classical Boltzmann scattering maps $S:\partial\mathcal{P}_{2}(\mathsf{P}_{\ast})\times\mathbb{R}^{6}\rightarrow\mathbb{R}^{6}$. There exist initial spatial configurations $X_{0}=[x_{0}, \ov{x}_{0}, \vartheta_{0}, \ov{\vartheta}_{0}]\in \partial\mathcal{P}_{2}(\mathsf{P}_{\ast})$, initial velocity vectors $V_{0}=[v_{0}, \ov{v}_{0}, \omega_{0}, \ov{\omega}_{0}]\in\Sigma^{0}_{X_{0}}\subset\mathbb{R}^{6}$, and $\delta=\delta(X_{0})>0$ such that
\begin{equation*}
\mathscr{L}_{2}(\mathsf{P}(t)\cap\ov{\mathsf{P}}(t))>0 \quad \text{for}\hspace{2mm} -\delta<t<0 \quad \text{and}\quad 0<t<\delta,
\end{equation*}
where $t\mapsto X(t)$ is given by
\begin{equation*}
X(t):=\left[
\begin{array}{c}
x_{0}+tv_{0} \\
\ov{x}_{0}+t\ov{v}_{0}\\
\vartheta_{0}+t\omega_{0}\\
\ov{\vartheta}_{0}+t\ov{\omega}_{0}
\end{array}
\right] \quad \text{for}-\delta <t<\delta.
\end{equation*}
\end{cor}
As such, the initial boundary value problem for hard ellipses is manifestly very different to that for hard disks. We discuss the ramifications of theorem \ref{mainresult} for the existence of `rough' solutions of \eqref{masta} briefly in the final section \ref{closrem}.
\section{Pre- and Post-collisional Velocities} In order to understand how to assign boundary conditions for a dynamics $t\mapsto X(t)$ on $\mathcal{P}_{2}(\mathsf{P}_{\ast})$ in the greatest possible generality, let us now parameterise $\partial\mathcal{P}_{2}(\mathsf{P}_{\ast})$ in a convenient manner.
\subsection{Parameterisation of Collision Configurations}\label{parz}
As $d=2$ in all the sequel, it will be helpful to identify any member $R\in\mathrm{SO}(2)$ with the angle of rotation $\theta\in\mathbb{S}^{1}$ (modulo $2\pi$) to which it gives rise. In particular, $R\in\mathrm{SO}(2)$ if and only if there exists $\theta\in\mathbb{S}^{1}$ such that
\begin{equation}
R=R(\theta)=\left(
\begin{array}{cc}
\cos\theta & -\sin\theta \\
\sin\theta & \cos\theta
\end{array}
\right).
\end{equation}
Without loss of generality, the set $\mathsf{P}_{\ast}$ can be regarded as fixed at the origin. Since the set $\mathsf{P}_{\ast}$ is compact, any collision configuration $X\in\partial\mathcal{P}_{d}(\mathsf{P}_{\ast})$ admits the representation
\begin{equation}\label{boundbeta}
X=\left[
\begin{array}{c}
y \\ 
y+d_{\beta}e(\psi) \\
0 \\
\theta
\end{array}
\right]
\end{equation}
for some $y\in\mathbb{R}^{2}$ and $\beta=(\theta, \psi)\in\mathbb{T}^{2}$, where
\begin{equation}
e(\psi):=\left(
\begin{array}{c}
\cos\psi \\
\sin\psi
\end{array}
\right)
\end{equation}
and $d_{\beta}>0$ is the {\em distance of closest approach} defined by
\begin{equation}
d_{\beta}:=\inf\left\{
d>0\,:\,\mathscr{L}_{2}(\mathsf{P}_{\ast}\cap(R(\theta)\mathsf{P}_{\ast}+de(\psi)))=0
\right\}.
\end{equation}
Thus, if one fixes $y$ to be the origin, a given $X\in\partial\mathcal{P}_{2}(\mathsf{P}_{\ast})$ is characterised by an element of the 2-torus. It will also be helpful in what follows to define the {\em collision vector} $p_{\beta}\in\mathbb{R}^{2}$ given by the unique element of the set
\begin{equation}
\mathsf{P}_{\ast}\cap(R(\theta)\mathsf{P}_{\ast}+d_{\beta}e(\psi)),
\end{equation}
and its associated {\em conjugate collision vector} $q_{\beta}:=d_{\beta}e(\psi)-p_{\beta}$. Finally, we denote by $n_{\beta}\in\mathbb{R}^{2}$ the unique outward-pointing unit normal to the boundary of the set $\mathsf{P}_{\ast}$ at the point $p_{\beta}\in\partial\mathsf{P}_{\ast}$.
\subsection{Basic Definitions of Velocity Maps}
Let us begin by defining the most general class of pre-collisional velocities with which we shall work in all the sequel.
\begin{defn}[Regular Pre-collisional Velocity Maps]\label{pre}
Suppose a collision configuration $\beta\in\mathbb{T}^{2}$ is given and fixed. We define the class of all {\bf regular pre-collisional velocity maps $\widetilde{\Sigma}_{\beta}^{-}(\mathsf{P}_{\ast})$ with respect to} $\beta$ to be the set of locally-integrable maps
\begin{equation*}
\widetilde{\Sigma}_{\beta}^{-}(\mathsf{P}_{\ast}):=\left\{
V\in L^{1}_{\mathrm{loc}}(\mathbb{R}, \mathbb{R}^{6})\,:\,\mathscr{L}_{2}(\mathsf{P}(t)\cap\ov{\mathsf{P}}(t))=0\hspace{1mm}\hspace{1mm}\text{for all}\hspace{2mm} t\leq 0.
\right\},
\end{equation*}
where the set evolutions $t\mapsto\mathsf{P}(t)$ and $t\mapsto\ov{\mathsf{P}}(t)$ are given by \eqref{setevo} above, and the dynamics $t\mapsto X(t)$ is defined componentwise by
\begin{equation}\label{com}
x(t):=\int_{0}^{t}v(s)\,ds \quad \text{and}\quad \ov{x}(t):=d_{\beta}e(\psi)+\int_{0}^{t}\ov{v}(s)\,ds
\end{equation}
and also
\begin{equation}\label{ori}
\vartheta(t):=\vartheta+\int_{0}^{t}\omega(s)\,ds \quad \text{and}\quad \ov{\vartheta}(t):=\ov{\vartheta}+\int_{0}^{t}\ov{\omega}(s)\,ds.
\end{equation}
\end{defn}
One notable feature of our definition of pre-collisional velocities is that one cannot speak meaningfully of the pointwise value in $\mathbb{R}^{6}$ of the velocity map $V(\tau)$ at the collision time $\tau=0$. As such, in our present framework the decision problem in section \ref{decision} has no immediate sense, and this na\"{i}ve approach must be superseded by one more sophisticated. Indeed, we spend our subsequent efforts in the refinements of {\bf (Q1)} and {\bf (Q2)} in sections below.
\begin{rem}\label{browny}
As general as its definition might seem, $\widetilde{\Sigma}_{\beta}^{-}(\mathsf{P}_{\ast})$ does not contain those velocities $V$ whose associated spatial maps are sample paths of a \emph{brownian motion} on $\mathcal{P}_{2}(\mathsf{P}_{\ast})$. This is a `large' class of motions on $\mathcal{P}_{2}(\mathsf{P}_{\ast})$ that one might deem to be natural. The reason we exclude such velocities in this work is that they cannot be localised pointwise in time. Indeed, this will be of importance to us when we come to employ function germs in section \ref{germs} below.

This lack of localisability is readily seen. Indeed, let $\Xi$ be a non-empty set, $\mathcal{S}\subseteq 2^{\Xi}$ be a sigma algebra, and $\mathbb{P}$ a measure on $\mathcal{S}$. For the given measure space $(\Xi, \mathcal{S}, \mathds{P})$, suppose one is given\footnote{Suppose that $N\gg 1$, and that $\{\mathsf{P}_{j}\}_{j=1}^{N}$ is a family of congruent ellipsoids. To our knowledge, no explicit construction of a $\mathcal{P}_{d}(\mathsf{P}_{1}, ..., \mathsf{P}_{N})$-valued Brownian motion on a measure space $(\Xi, \mathcal{S}, \mathds{P})$ has been performed in the literature, when $\mathsf{P}_{i}$ are only of class $C^{0}$. As it is related to this issue, the reader might wish to consider the notes of \textsc{Varadhan} \cite{vara} for the construction of reflected brownian motions in certain geometrically-simple domains.} a brownian motion $\{W(\cdot, t)\}_{t\in\mathbb{R}}$ thereon, for which $W(\xi, \cdot):\mathbb{R}\rightarrow\mathcal{P}_{2}(\mathsf{P}_{\ast})$ for $\xi\in\Xi$. It is well known (see, for instance, \textsc{Evans} \cite{evans2013introduction} for the construction of a Brownian motion which is suitable for collision-free dynamics) that for $\mathbb{P}$-almost every $\xi\in\Omega$, the sample path $t\mapsto W(\xi, t)$ is of class $L^{1}_{\mathrm{loc}}(\mathbb{R}, \mathcal{P}_{2}(\mathsf{P}_{\ast}))$ but is {\em not} of bounded variation on $\mathbb{R}$. As such, its distributional derivative $V_{\xi}:= DW(\xi, \cdot)$ is not even an $\mathbb{R}^{6}$-valued Borel measure on $\mathbb{R}$. Manifestly, one cannot speak of the pointwise value of $V_{\xi}$ almost everywhere on $\mathbb{R}$. This property of localisability is important when we come to formalise velocities as germs.
\end{rem}
With this brief discussion on the reasonableness of definition \ref{pre} of pre-collisional velocities for two rigid sets congruent to $\mathsf{P}_{\ast}$ at collision in place, let us set out the remaining classes of velocities we employ in this work.
\begin{defn}[Regular Post-collisional Velocity Maps]
Suppose a collision configuration $\beta\in\mathbb{T}^{2}$ is given and fixed. We define the class of all {\bf regular post-collisional velocity maps $\widetilde{\Sigma}_{\beta}^{-}(\mathsf{P}_{\ast})$ with respect to} $\beta$ to be 
\begin{equation*}
\widetilde{\Sigma}_{\beta}^{+}(\mathsf{P}_{\ast}):=\left\{
V\in L^{1}_{\mathrm{loc}}(\mathbb{R}, \mathbb{R}^{6})\,:\,\mathscr{L}_{2}(\mathsf{P}(t)\cap\ov{\mathsf{P}}(t))=0\hspace{1mm}\hspace{1mm}\text{for all}\hspace{2mm} t\geq 0.
\right\},
\end{equation*}
\end{defn}
If an external body force acts on the sets $\mathsf{P}$ and $\ov{\mathsf{P}}$, or their boundaries are endowed with additional structure, it is possible for two such sets to admit complex behaviour in the neighbourhood of a collision time $\tau$, including so-called {\em grazing motions} or {\em sticky motions}; we refer the reader to \textsc{Bressan and Nguyen} \cite{bressan2013non} for a discussion on such boundary conditions. In particular, the sets $\mathsf{P}(t)$ and $\ov{\mathsf{P}}(t)$ may remain in contact with one another for an arbitrarily-long time following an initial collision at $t=\tau$. For this reason, we establish the following definition.
\begin{defn}[Regular Grazing Velocity Maps]
Suppose a collision configuration $\beta\in\mathbb{T}^{2}$ is given and fixed. We define the class of all {\bf regular grazing velocity maps $\widetilde{\Sigma}_{\beta}^{0}(\mathsf{P}_{\ast})$ with respect to} $\beta$ to be
\begin{equation*}
\widetilde{\Sigma}_{\beta}^{0}(\mathsf{P}_{\ast}):=\left\{
V\in L^{1}_{\mathrm{loc}}(\mathbb{R}, \mathbb{R}^{6})\,:\,\mathscr{L}_{2}(\mathsf{P}(t)\cap\ov{\mathsf{P}}(t))=0\hspace{1mm}\hspace{1mm}\text{for all}\hspace{2mm} t\in\mathbb{R}
\right\}.
\end{equation*}
\end{defn}
Clearly, one has that $\widetilde{\Sigma}_{\beta}^{0}(\mathsf{P}_{\ast})=\widetilde{\Sigma}_{\beta}^{-}(\mathsf{P}_{\ast})\cap \widetilde{\Sigma}_{\beta}^{+}(\mathsf{P}_{\ast})$. Finally, with question {\bf (Q2)} above in mind, we are motivated to define the following class of velocities.
\begin{defn}[Regular Inadmissible Velocity Maps]
Suppose a collision configuration $\beta\in\mathbb{T}^{2}$ is given and fixed. We define the class of all {\bf regular inadmissible velocity maps $\widetilde{\Sigma}_{\beta}^{\times}(\mathsf{P}_{\ast})$ with respect to} $\beta$ to be
\begin{equation*}
\widetilde{\Sigma}_{\beta}^{\times}(\mathsf{P}_{\ast}):=\left\{
V\in L^{1}_{\mathrm{loc}}(\mathbb{R}, \mathbb{R}^{6})\,:\,\exists\,\mathcal{N}\subset\mathbb{R}\hspace{2mm}\text{s.t.} \hspace{2mm}\mathscr{L}_{2}(\mathsf{P}(t)\cap\ov{\mathsf{P}}(t))>0\hspace{2mm}\text{for}\hspace{1mm} t\in\mathcal{N}
\right\},
\end{equation*}
where $\mathcal{N}\subset\mathbb{R}$ ranges over the class of open intervals containing the origin.
\end{defn}
With the above definitions in place, we record the following simple observation.
\begin{lem}
For any $\beta\in\mathbb{T}^{2}$, one has that
\begin{equation*}
L^{1}_{\mathrm{loc}}(\mathbb{R}, \mathbb{R}^{6})=\bigcup_{\Box\in\{-, +, 0, \times\}}\widetilde{\Sigma}_{\beta}^{\Box}(\mathsf{P}_{\ast}).
\end{equation*}
\end{lem}
As the above definitions stand, for our purposes a given $V\in L^{1}_{\mathrm{loc}}(\mathbb{R}, \mathbb{R}^{6})$ contains `too much' information about a collision at $\tau=0$, in the sense that we are only interested in the behaviour of the maps \eqref{com} in a neighbourhood of the collision time, not on the whole real line. Such redundant information in $L^{1}_{\mathrm{loc}}(\mathbb{R}, \mathbb{R}^{6})$ is conveniently factored out by employing the notion of germ. 
\subsection{Germs of Velocities}\label{germs}
To fix ideas, we focus on the class of regular pre-collisional velocity maps. We note that although a given $V\in\widetilde{\Sigma}_{\beta}^{-}(\mathsf{P}_{\ast})$ is defined globally (almost everywhere) on the negative half-line, for the purposes of describing a collision between $\mathsf{P}(t)$ and $\ov{\mathsf{P}}(t)$ at $t=0$ we are only truly interested in the local behaviour of a velocity map $V$ in a left-neighbourhood of 0. Indeed, we would like to understand two given velocities $V, W\in L^{1}_{\mathrm{loc}}(\mathbb{R}, \mathbb{R}^{6})$ as giving rise to the same object if and only if their restrictions to some left-neighbourhood of 0 coincide. The concept of {\em germ} is appropriate here: indeed, see \textsc{Warner} \cite{warner2013foundations} for the definition of germ when $V, W\in C^{\infty}(\mathbb{R}, \mathbb{R}^{6})$. We recall the following for the convenience of the reader.
\begin{defn}[Left-germs and Right-germs of $L^{1}_{\mathrm{loc}}(\mathbb{R}, \mathbb{R}^{6})$ Maps]\label{lrgerms}
For a given $V\in L^{1}_{\mathrm{loc}}(\mathbb{R}, \mathbb{R}^{6})$ and $\tau\in\mathbb{R}$, we write $[V]_{\tau}^{-}$ to denote the class of maps
\begin{equation*}
[V]_{\tau}^{-}:=\left\{
W\in L^{1}_{\mathrm{loc}}(\mathbb{R}, \mathbb{R}^{6})\,:\,\exists \mathcal{N}^{-}_{\tau}\subset\mathbb{R}\hspace{2mm}\text{s.t.}\hspace{2mm} W|_{\mathcal{N}^{-}}=V|_{\mathcal{N}^{-}}\hspace{2mm}\text{in}\hspace{2mm}L^{1}(\mathcal{N}^{-}, \mathbb{R}^{6})
\right\},
\end{equation*}
where $\mathcal{N}^{-}_{\tau}=(\tau-\delta, \tau]$ for some $\delta>0$. We call $[V]_{\tau}^{-}$ the {\bf left-germ} at $\tau$ generated by $V$, and write
\begin{equation*}
G^{-}[\tau]:=\left\{
[V]_{\tau}^{-}\,:\,V\in L^{1}_{\mathrm{loc}}(\mathbb{R}, \mathbb{R}^{6})
\right\}.
\end{equation*} Similarly, we write $[V]_{\tau}^{+}$ to denote
\begin{equation*}
[V]_{\tau}^{+}:=\left\{
W\in L^{1}_{\mathrm{loc}}(\mathbb{R}, \mathbb{R}^{6})\,:\,\exists \mathcal{N}^{+}_{\tau}\subset\mathbb{R}\hspace{2mm}\text{s.t.}\hspace{2mm} W|_{\mathcal{N}^{+}}=V|_{\mathcal{N}^{+}}\hspace{2mm}\text{in}\hspace{2mm}L^{1}(\mathcal{N}^{+}, \mathbb{R}^{6})
\right\},
\end{equation*}
where $\mathcal{N}^{+}_{\tau}=[\tau, \tau+\delta)$ for some $\delta>0$. We call $[V]_{\tau}^{-}$ the {\bf right-germ} at $\tau$ generated by $V$, and write
\begin{equation*}
G^{+}[\tau]:=\left\{
[V]_{\tau}^{+}\,:\,V\in L^{1}_{\mathrm{loc}}(\mathbb{R}, \mathbb{R}^{6})
\right\}.
\end{equation*}
\end{defn}
In order to understand grazing and inadmissible velocities, we also establish the following definition.
\begin{defn}[Germs of $L^{1}_{\mathrm{loc}}(\mathbb{R}, \mathbb{R}^{6})$ Maps]\label{germs}
For a given $V\in L^{1}_{\mathrm{loc}}(\mathbb{R}, \mathbb{R}^{6})$ and $\tau\in\mathbb{R}$, we write $[V]_{\tau}$ to denote the class of maps
\begin{equation*}
[V]_{\tau}:=\left\{
W\in L^{1}_{\mathrm{loc}}(\mathbb{R}, \mathbb{R}^{6})\,:\,\exists \mathcal{N}_{\tau}\subset\mathbb{R}\hspace{2mm}\text{s.t.}\hspace{2mm} W|_{\mathcal{N}}=V|_{\mathcal{N}}\hspace{2mm}\text{in}\hspace{2mm}L^{1}(\mathcal{N}, \mathbb{R}^{6})
\right\},
\end{equation*}
where $\mathcal{N}_{\tau}=(\tau-\delta, \tau+\delta)$ for some $\delta>0$. We call $[V]_{\tau}$ the {\bf germ} at $\tau$ generated by $V$, and denote
\begin{equation*}
G[\tau]:=\left\{
[V]_{\tau}\,:\,V\in L^{1}_{\mathrm{loc}}(\mathbb{R}, \mathbb{R}^{6})
\right\}.
\end{equation*}
\end{defn}
We now introduce the spaces of velocity germ with which we work in the sequel.
\begin{defn}[Collision Velocity Germs at 0]
For a given $\beta\in\mathbb{T}^{2}$, we define the following sets of (equivalence classes of) velocities:
\begin{align}
G_{\beta}^{-}(\mathsf{P}_{\ast}):=\left\{
[V]_{0}^{-}\,:\,V\in\widetilde{\Sigma}_{\beta}^{-}(\mathsf{P}_{\ast})
\right\}, \vspace{2mm} \label{pr}\\
G_{\beta}^{+}(\mathsf{P}_{\ast}):=\left\{
[V]_{0}^{+}\,:\,V\in\widetilde{\Sigma}_{\beta}^{+}(\mathsf{P}_{\ast})
\right\}, \vspace{2mm} \label{ps} \\
G_{\beta}^{0}(\mathsf{P}_{\ast}):=\left\{
[V]_{0}\,:\,V\in\widetilde{\Sigma}_{\beta}^{0}(\mathsf{P}_{\ast})
\right\}, \vspace{2mm} \label{pp}\\
G_{\beta}^{\times}(\mathsf{P}_{\ast}):=\left\{
[V]_{0}\,:\,V\in\widetilde{\Sigma}_{\beta}^{\times}(\mathsf{P}_{\ast}) 
\right\}. \label{inad}
\end{align}
\end{defn}
We note that these sets $G^{\Box}_{\beta}(\mathsf{P}_{\ast})$ of collision velocity germs at 0 do not, in general, admit the structure of a real vector space. We shall ultimately find it useful to restrict our attention to vector subspaces of $L^{1}_{\mathrm{loc}}(\mathbb{R}, \mathbb{R}^{6})$ of piecewise smooth velocity maps.
\begin{rem}
We do not take steps to try to generalise the concept of germ from the class of regular distributions $L^{1}_{\mathrm{loc}}(\mathbb{R}, \mathbb{R}^{6})\hookrightarrow\mathcal{D}(\mathbb{R}, \mathbb{R}^{6})'$ to non-regular distributions in $\mathcal{D}(\mathbb{R}, \mathbb{R}^{6})'$. It is for this reason our framework does not include the distributional derivatives of sample paths of brownian motions, for instance.
\end{rem}
\subsection{Scattering Maps and Initial Boundary Value Problems for Set Dynamics}\label{scatteringmaps}
We claimed above that the decision problem of section \ref{decision} is pertinent because, in the study of set dynamics, one is interested in the construction of a map $\sigma_{\beta}: G_{\beta}^{-}(\mathsf{P}_{\ast})\rightarrow G_{\beta}^{+}(\mathsf{P}_{\ast})$ which uniquely assigns a post-collisional velocity to any given pre-collisional velocity. Such a {\bf scattering map} $\sigma_{\beta}$ is an essential ingredient in the construction of global-in-time solutions -- in whichever appropriate sense -- of equations of motion governing $\mathsf{P}(t)$ and $\ov{\mathsf{P}}(t)$ that arise in many domains of interest. Given our interest in kinetic theory, we fix our domain of interest here to those `physical' equations of motion which arise in the field of classical mechanics. In particular, we appeal to Euler's First and Second Laws of motion: see \textsc{Truesdell} \cite{truesdell2016first} for the quasi-axiomatic approach to mechanics which yields these basic equations. 

\subsubsection{An Initial Boundary Value Problem} 
Suppose that $F:\mathbb{R}^{6}\times \mathbb{R}\rightarrow\mathbb{R}^{6}$ is a locally integrable map. We seek a continuous map $X: \mathbb{R}\rightarrow\mathcal{P}_{2}(\mathsf{P}_{\ast})$ satisfying the system
\begin{equation}\label{newtons}
\textbf{(IBVP)}\quad \left\{
\begin{array}{l}
M\ddot{X}(t)=F(X(t), t)\quad \text{for}\hspace{2mm}X(t)\in\mathrm{int}\,\mathcal{P}_{2}(\mathsf{P}_{\ast}), \vspace{2mm} \\
X(0)=X_{0}, \hspace{2mm} \dot{X}(0)=V_{0},
\end{array}
\right.
\end{equation}
where $\mathrm{int}\,\mathcal{P}_{2}(\mathsf{P}_{\ast}):=\mathcal{P}_{2}(\mathsf{P}_{\ast})\setminus \partial\mathcal{P}_{2}(\mathsf{P}_{\ast})$, and 
\begin{align}
M:=\mathrm{diag}(m, m, m, m, J, J)\in\mathbb{R}^{6\times 6},
\end{align}
with Newton's Second Law to be understood in the sense of distributions on $\mathbb{R}$. Firstly, we note that {\bf (IBVP)} is underdetermined in its present form as no boundary conditions for the dynamics on $\partial\mathcal{P}_{2}(\mathsf{P}_{\ast})$ has been provided. Secondly, if $X$ is only supposed to be continuous, one may not be able to understand the initial condition $\dot{X}(0)=V_{0}$ pointwise in the classical sense (c.f. remark \ref{browny}). Thirdly, given that one expects collisions to occur, $\ddot{X}$ should be understood as a vector-valued Radon measure. It is for these reasons that one ought to (i) describe the values of $\dot{X}$ on $\partial\mathcal{P}_{2}(\mathsf{P}_{\ast})$ in terms of velocity germs, and (ii) work with the notion of scattering map in order to endow the above problem with a suitable notion of boundary conditions.
\subsubsection{Scattering Maps}\label{scatty}
To clarify what we mean by a {\em solution} of {\bf (IBVP)} above, we now set out the following definition.
\begin{defn}[Scattering Maps]
Suppose $\beta\in\mathbb{T}^{2}$ is given and fixed. Any injection $\sigma_{\beta}: G_{\beta}^{-}(\mathsf{P}_{\ast})\rightarrow G_{\beta}^{+}(\mathsf{P}_{\ast})$ is said to be a {\bf scattering map} associated to the collision configuration $\beta$. We call the collection $\{\sigma_{\beta}\}_{\beta\in\mathbb{T}^{2}}$ a {\em family of scattering maps}.
\end{defn}
Our main result concerns families of {\em frictionless} scattering maps (a class to which the classical Boltzmann scattering matrices belong).
\begin{defn}[Frictionless Scattering Maps]
We say that a scattering map $\sigma_{\beta}:G_{\beta}^{-}(\mathsf{P}_{\ast})\rightarrow G_{\beta}^{+}(\mathsf{P}_{\ast})$ is {\bf frictionless} if and only if $\sigma_{\beta}|_{G^{0}_{\beta}(\mathsf{P}_{\ast})}=\mathrm{id}$ on $G^{0}_{\beta}(\mathsf{P}_{\ast})$.
\end{defn}
Without specifying any additional criteria, it is clear there are a great many scattering maps $\sigma_{\beta}$ defined on $G_{\beta}^{-}(\mathsf{P}_{\ast})$. In the absence of external forcing ($F\equiv 0$), it is natural that one specify additionally that each $\sigma_{\beta}$ respect the conservation laws of classical mechanics. Specifically, for any $V\in[V]_{0}^{-}$ and $W\in\sigma_{\beta}([V]_{0}^{-})$, there exists $\delta=\delta(V, W)>0$ such that the conservation of linear momentum 
\begin{equation}\label{wcolm}
\int_{-\infty}^{\infty}\widehat{E}_{j}\cdot V(t)\varphi(t)\,dt=\int_{-\infty}^{\infty}\widehat{E}_{j}\cdot W(t)\varphi(t)\,dt, \tag{{\em w}COLM}
\end{equation}
the conservation of angular momentum
\begin{equation}\label{wcoam}
\int_{-\infty}^{\infty}\Gamma_{\beta}(t)\cdot V(t)\varphi(t)\,dt=\int_{-\infty}^{\infty}\Gamma_{\beta}(0)\cdot W(t)\varphi(t)\,dt \tag{{\em w}COAM}
\end{equation}
and the conservation of kinetic energy
\begin{equation}\label{wcoke}
\int_{-\infty}^{\infty}|MV(t)|^{2}\varphi(t)\,dt=\int_{-\infty}^{\infty}|MW(t)|^{2}\varphi(t)\,dt \tag{{\em w}COKE}
\end{equation}
hold for all $\varphi\in C^{\infty}_{c}((-\delta, \delta), \mathbb{R})$, where
\begin{equation}
\widehat{E}_{1}:=\left(
\begin{array}{c}
1 \\
0 \\
1 \\
0 \\
0 \\
0
\end{array}
\right), \quad \widehat{E}_{2}:=\left(
\begin{array}{c}
0 \\
1 \\
0 \\
1 \\
0 \\
0
\end{array}
\right), \quad \Gamma_{\beta}:=\left(
\begin{array}{c}
0 \\
0 \\
-md_{\beta}\sin\psi \\
md_{\beta}\cos\psi \\
J \\
J 
\end{array}
\right),
\end{equation}
and $M\in\mathbb{R}^{6\times 6}$ is the mass-inertia tensor
\begin{equation}
M:=\mathrm{diag}(\sqrt{m}, \sqrt{m}, \sqrt{m}, \sqrt{m}, \sqrt{J}, \sqrt{J}).
\end{equation}
Indeed, we shall call any scattering map $\sigma_{\beta}$ that satisfies \eqref{wcolm}, \eqref{wcoam} and \eqref{wcoke} on $G_{\beta}^{-}(\mathsf{P}_{\ast})$ {\bf physical}. For instance, it can be shown that for each $\beta\in\mathbb{T}^{2}$, the Boltzmann matrices
\begin{equation}\label{bm}
s_{\beta}:=M^{-1}\left(
I-2\widehat{\nu}_{\beta}\otimes \widehat{\nu}_{\beta}
\right)M\in\mathrm{O}(6),
\end{equation}
where
\begin{equation}
\widehat{\nu}_{\beta}:=\frac{1}{\sqrt{\frac{2}{m}++\frac{1}{J}|p_{\beta}^{\perp}\cdot n_{\beta}|^{2}+\frac{1}{J}|q_{\beta}^{\perp}\cdot n_{\beta}|^{2}}}M^{-1}\left[
\begin{array}{c}
-n_{\beta} \\
n_{\beta} \\
-p_{\beta}^{\perp}\cdot n_{\beta} \\
q_{\beta}^{\perp}\cdot n_{\beta}
\end{array}
\right]
\end{equation}
give rise to physical scattering maps on the subspace of $G_{\beta}^{-}(\mathsf{P}_{\ast})$ corresponding to germs generated by those maps which are locally constant at 0, namely $\mathcal{F}(\mathbb{R}, \mathbb{R}^{6})$. This leads us naturally to a discussion on subspaces of $G_{\beta}^{\Box}(\mathsf{P}_{\ast})$ for $\Box\in\{-, +, 0, \times\}$ in section \ref{smoothy} below.
\begin{rem}
While each $G_{\beta}^{\Box}(\mathsf{P}_{\ast})$ is a subset of a real vector space, it is not immediately clear what constitutes a `natural' topology for these spaces of velocities. Indeed, it is necessary to understand $G_{\beta}^{\Box}(\mathsf{P}_{\ast})$ as subsets of topological vector spaces in order to establish continuity -- or perhaps even smoothness, in the case when the topology is norm-induced and Banach -- of scattering maps $\sigma_{\beta}$ defined thereon. Moreover, one would like to be able to understand the map
\begin{equation}
\beta\mapsto \sigma_{\beta}(V) \quad \text{for a fixed}\hspace{2mm}V\in\bigcap_{\beta\in B}G^{-}_{\beta}(\mathsf{P}_{\ast})
\end{equation}
as continuous on those open subsets $B\subseteq\mathbb{T}^{2}$ for which the above intersection is non-empty. This would require one to construct a `natural' topology on the union
\begin{equation}
\bigcup_{\beta\in B}G_{\beta}^{+}(\mathsf{P}_{\ast}).
\end{equation}
In section below, we are able to make sense of all the above when restricting our attention to subspaces of $G_{\beta}^{\Box}(\mathsf{P}_{\ast})$ of germs of {\em smooth} functions.
\end{rem}
\subsubsection{Formal Definition of a Solution to {\bf (IBVP)}}
Let us begin by formalising the definition of the set of {\em collision times} for the dynamics of two sets.
\begin{defn}[Collision Times]
Suppose $X:\mathbb{R}\rightarrow\mathcal{P}_{2}(\mathsf{P}_{\ast})$ is continuous. The set of {\em collision times} $\mathcal{T}(X)\subseteq\mathbb{R}$ associated to $X$ is defined to be 
\begin{equation}
\mathcal{T}(X):=\left\{
t\in\mathbb{R}\,:\,\mathsf{P}(t)\cap\ov{\mathsf{P}}(t)\neq\varnothing
\right\}.
\end{equation}
\end{defn}
We now state precisely what we mean by a solution of the above initial boundary value problem. In what follows, we identify vectors $V_{0}\in\mathbb{R}^{6}$ with the constant maps on $\mathbb{R}$ to which they give rise.
\begin{defn}[Weak Solutions of {\bf (IBVP)}]\label{formaldeffy}
Suppose $\mathsf{P}_{\ast}\in\mathcal{C}$, that $F:\mathbb{R}^{6}\times\mathbb{R}\rightarrow\mathbb{R}^{6}$ is locally integrable, and that $\{\sigma_{\beta}\}_{\beta\in\mathbb{T}^{2}}$ is a family of scattering maps. Let $X_{0}\in \mathcal{P}_{2}(\mathsf{P}_{\ast})$ and
\begin{equation*}
V_{0}\in\left\{
\begin{array}{ll}
G[0]& \quad \text{if}\hspace{2mm}X_{0}\in\mathrm{int}\,\mathcal{P}_{2}(\mathsf{P}_{\ast}),\vspace{2mm}\\
G_{\beta}^{-}(\mathsf{P}_{\ast}) & \quad \text{if}\hspace{2mm}X_{0}\in \partial\mathcal{P}_{2}(\mathsf{P}_{\ast}).
\end{array}
\right.
\end{equation*}
We say that $X\in C^{0}(\mathbb{R}, \mathcal{P}_{2}(\mathsf{P}_{\ast}))$ is a global-in-time weak solution of
\begin{equation}\label{newtonforcey}
M\ddot{X}(t)=F(X(t), t)
\end{equation}
with initial state $Z_{0}:=[X_{0}, V_{0}]$ if and only if $\dot{X}\in \mathrm{BV}_{\mathrm{loc}}(\mathbb{R}, \mathbb{R}^{6})$ and $X$ satisfies
\begin{equation*}
\int_{-\infty}^{\infty}\left(X(t)-F(X(t), t)\right)\cdot\phi''(t)\,dt=\int_{-\infty}^{\infty}\phi(t)\,d\mu(t)
\end{equation*}
for all $\phi\in C^{\infty}_{c}(\mathbb{R}, \mathbb{R}^{6})$, where $\mu$ is a $\mathbb{R}^{6}$-valued Radon measure supported on $\mathcal{T}(X)$. Moreover, for all $\tau\in\mathcal{T}(X)$,
\begin{equation*}
[\dot{X}]_{\tau}^{-}\in G_{\beta(\tau)}^{-}(\mathsf{P}_{\ast}) \quad \text{and}\quad [\dot{X}]_{\tau}^{+}= \sigma_{\beta(\tau)}([\dot{X}]_{\tau}^{-}),
\end{equation*}
where $\beta(\tau)$ is determined by \eqref{boundbeta}. Finally, $X(0)=X_{0}$ and
\begin{equation*}
[\dot{X}]_{0}=V_{0} \quad \text{if}\hspace{2mm}V_{0}\in G[0],
\end{equation*}
or
\begin{equation*}
[\dot{X}]_{0}^{-}=V_{0} \quad \text{if}\hspace{2mm}V_{0}\in G_{\beta}^{-}(\mathsf{P}_{\ast}).
\end{equation*}
\end{defn}
\begin{rem}
This paper is {\em not} concerned with the construction of global-in-time weak solutions of {\bf (IBVP)} for arbitrary initial data. Rather, our main result \ref{mainresult} is concerned with the characterisation of both $G_{\beta}^{-}(\mathsf{P}_{\ast})$ and $G_{\beta}^{+}(\mathsf{P}_{\ast})$ so that one might, in turn, construct a scattering map $\sigma_{\beta}$ defined thereon. As we have already claimed, this is an essential datum required in many methods for the construction of solutions to {\bf (IBVP)}. We refer the reader to \textsc{Ballard} \cite{ballard2000dynamics}, \textsc{Raous, Jean and Moreau} \cite{raous1995contact} and \textsc{Schatzman} \cite{schatzman1978class} for the r\^{o}le of scattering maps in existence theories for set dynamics.
\end{rem}
\subsection{Subsets of $G_{\beta}^{\Box}(\mathsf{P}_{\ast})$ of `Smooth' Velocity Maps}\label{smoothy}
The sets $G^{\Box}_{\beta}(\mathsf{P}_{\ast})$ contain those germs whose members are only known to be locally integrable at 0. Depending on the {\bf (IBVP)} of interest, one might wish to establish the existence of solutions which admit some regularity properties. For this reason, we introduce the notion of subsets of $G^{\Box}_{\beta}(\mathsf{P}_{\ast})$ of $\mathcal{X}$-germs.
\begin{defn}[$\mathcal{X}$-germs]
Suppose that $\mathcal{X}(\mathbb{R}, \mathbb{R}^{6})$ is a non-empty vector subspace of $L^{1}_{\mathrm{loc}}(\mathbb{R}, \mathbb{R}^{6})$. We define the $\mathcal{X}$-germs $G_{\beta}^{\Box}(\mathsf{P}_{\ast}; \mathcal{X})$ for $\Box\in\{-, +, 0, \times\}$ by
\begin{equation*}
G_{\beta}^{\Box}(\mathsf{P}_{\ast}; \mathcal{X}):=\left\{
[V]_{0}^{\Box}\,: V\in\mathcal{X}(\mathbb{R}, \mathbb{R}^{6})
\right\}\quad \text{for}\hspace{2mm}\Box\in\{-, +\}
\end{equation*} 
and also
\begin{equation*}
G_{\beta}^{\Box}(\mathsf{P}_{\ast}; \mathcal{X}):=\left\{
[V]_{0}\,: V\in\mathcal{X}(\mathbb{R}, \mathbb{R}^{6})
\right\}\quad \text{for}\hspace{2mm}\Box\in\{0, \times\}.
\end{equation*} 
\end{defn}
Natural choices for $\mathcal{X}(\mathbb{R}, \mathbb{R}^{6})$ include maps of locally-bounded variation $\mathrm{BV}_{\mathrm{loc}}(\mathbb{R}, \mathbb{R}^{6})$, smooth maps $C^{k}(\mathbb{R}, \mathbb{R}^{6})$ for $0\leq k<\infty$, infinitely-differentiable maps $C^{\infty}(\mathbb{R}, \mathbb{R}^{6})$ or real-analytic maps $C^{\omega}(\mathbb{R}, \mathbb{R}^{6})$. Intuitively, the smaller the subspace $\mathcal{X}$, the more tractable the characterisation of the associated $\mathcal{X}$-germs should become. Indeed, this is the case when $\mathcal{X}(\mathbb{R}, \mathbb{R}^{6})$ is chosen to be the set $\mathcal{F}(\mathbb{R}, \mathbb{R}^{6})$ of all constant maps. However, even in the relatively-simple case of constant maps, the structure of each $G^{\Box}_{\beta}(\mathsf{P}_{\ast}; \mathcal{X})$ is somewhat complicated and depends sensitively on the geometry of $\mathsf{P}_{\ast}\in\mathcal{C}$.
\subsection{Refinement of Questions (Q1) and (Q2)}
We are now in a position to restate questions {\bf (Q1)} and {\bf (Q2)} using the language of germs we have established over the previous sections. Suppose that $\mathcal{X}=\mathcal{X}(\mathbb{R}, \mathbb{R}^{6})$ is a vector subspace of $L^{1}_{\mathrm{loc}}(\mathbb{R}, \mathbb{R}^{6})$. We ask the following question with section \ref{decision} in mind:
\begin{enumerate}[{\bf(Q'1)}]
\item For the given collision configuration $\beta\in\mathbb{T}^{2}$, does a map $V\in L^{1}_{\mathrm{loc}}(\mathbb{R}, \mathbb{R}^{6})$ have the property that (i) $[V]_{0}^{-}\in \Sigma_{\beta}^{-}(\mathsf{P}_{\ast}; \mathcal{X})$, or (ii) $[V]_{0}^{+}\in \Sigma_{\beta}^{+}(\mathsf{P}_{\ast}; \mathcal{X})$?
\item For the given collision configuration $\beta\in\mathbb{T}^{2}$, is it the case that $\Sigma_{\beta}^{\times}(\mathsf{P}_{\ast}; \mathcal{X})\neq \varnothing$?
\end{enumerate}
In section \ref{harddeesk}, in the case of hard disks we shall answer {\bf (Q'2)} in the negative, whilst in section \ref{diskandell} in the case of hard ellipses we answer it in the affirmative. In line with the statement of \ref{mainresult}, we write
\begin{equation}
\mathcal{G}^{\Box}(\mathsf{P}_{\ast}; \mathcal{X}):=\left\{
G_{\beta}^{\Box}(\mathsf{P}_{\ast}; \mathcal{X})\,:\,\beta\in\mathbb{T}^{2}
\right\}
\end{equation}
to denote the class of all sets of velocity germs for a given $\Box\in\{-, +, 0, \times\}$.

As we are ultimately concerned with the particle model that underlies the Boltzmann equation, we study \eqref{newtonforcey} when there are no external body forces acting on the system, i.e. $F\equiv 0$ on $\mathbb{R}^{6}\times\mathbb{R}$. Even in the absence of body forces, it is possible for the dynamics to admit `pathological' behaviour. For this reason, we must refine definition \ref{formaldeffy} using the concept of {\em classical consistency}.
\begin{defn}[Classical Consistency Condition]
We say the members of the 1-parameter family 
\begin{equation}
\{\mu_{Z_{0}}\,:\,Z_{0}\in\mathcal{P}_{2}(\mathsf{P}_{\ast})\times\mathbb{R}^{6}\}
\end{equation}
of $\mathbb{R}^{6}$-valued Radon measures on $\mathbb{R}$ admit the {\bf classical consistency condition} if and only if for each $Z_{0}=[X_{0}, V_{0}]$, the associated map
\begin{equation}
X(t):=X_{0}+tV_{0}+\int_{0}^{t}\mu_{Z_{0}}([0, s])\,ds\quad \text{for}\hspace{2mm}t\in\mathbb{R}
\end{equation}
satisfies the ODE
\begin{equation}
M\ddot{X}(t)=0\quad \text{for all}\hspace{2mm}t\in I'
\end{equation}
pointwise in the classical sense on every non-empty open subinterval of the set of collision times $I'\subseteq\mathcal{T}(X)$ (whenever such subintervals exist).
\end{defn}
The purpose of the classical consistency condition is to rule out what we deem to be pathological behaviour of solutions of the IBVP associated to Newton's equations on $\mathcal{P}_{2}(\mathsf{P}_{\ast})$, namely those trajectories $t\mapsto X(t)$ which describe {\em rolling} phenomena. This will be discussed in greater detail in remark \ref{wtf} below. Let us now provide the notion of weak solution of Newton's equations with which we shall work in all the sequel.
\begin{defn}[Physical Weak Solutions of ({\bf IBVP})]\label{physweak}
Suppose $\mathsf{P}_{\ast}\in\mathcal{C}$, and that $\{\sigma_{\beta}\}_{\beta\in\mathbb{T}^{2}}$ is a family of physical scattering maps. Let $X_{0}\in\mathcal{P}_{2}(\mathsf{P}_{\ast})$ and $V_{0}\in\mathbb{R}^{6}$ satisfying 
\begin{equation*}
V_{0}\in\left\{
\begin{array}{ll}
G[0]& \quad \text{if}\hspace{2mm}X_{0}\in\mathrm{int}\,\mathcal{P}_{2}(\mathsf{P}_{\ast}),\vspace{2mm}\\
G_{\beta}^{-}(\mathsf{P}_{\ast}) & \quad \text{if}\hspace{2mm}X_{0}\in \partial\mathcal{P}_{2}(\mathsf{P}_{\ast}).
\end{array}
\right.
\end{equation*}
be given. We say that $X\in C^{0}(\mathbb{R}, \mathcal{P}_{2}(\mathsf{P}_{\ast}))$ is a {\bf physical local-in-time weak solution} of 
\begin{equation}\label{newty}
M\ddot{X}=0
\end{equation}
with initial state $Z_{0}:=[X_{0}, V_{0}]$ if and only if there exists an open interval $I\subset\mathbb{R}$ containing 0 such that
\begin{equation}
\int_{I}X(t)\cdot\phi''(t)\,dt=\int_{I}\phi(t)\,d\mu_{Z_{0}}(t) 
\end{equation}
for all $\phi\in C^{\infty}_{c}(I, \mathbb{R}^{6})$ for some $\mathbb{R}^{6}$-valued Radon measure $\mu_{Z_{0}}$ which admits the classical consistency condition. Moreover, for every $\tau\in\mathcal{T}(X)$, one has 
\begin{equation}
[\dot{X}]_{\tau}^{-}\in G_{\beta}^{-}(\mathsf{P}_{\ast}; \mathcal{F}) \quad \text{and}\quad [\dot{X}]_{\tau}^{+}=\sigma_{\beta}([X]_{\tau}^{-})\in G_{\beta}^{+}(\mathsf{P}_{\ast}; \mathcal{F}),
\end{equation}
In addition, the conservation laws \eqref{wcolm}, \eqref{wcoam} and \eqref{wcoke} hold true for all $\varphi\in C^{\infty}_{c}(I, \mathbb{R})$. Finally, $X$ agrees with the initial datum in the sense that $X(0)=X_{0}$ and 
\begin{equation*}
[\dot{X}]_{0}=V_{0} \quad \text{if}\hspace{2mm}V_{0}\in G[0],
\end{equation*}
or
\begin{equation*}
[\dot{X}]_{0}^{-}=V_{0} \quad \text{if}\hspace{2mm}V_{0}\in G_{\beta}^{-}(\mathsf{P}_{\ast}).
\end{equation*}
If $I$ can be chosen arbitrarily, we say that $X$ is a {\bf global-in-time weak solution} of \eqref{newty}.
\end{defn}
Once we show that $\mathcal{G}^{\times}(\mathsf{P}_{\ast}; \mathcal{F})\neq\{\varnothing\}$, it will follow immediately that there exist initial data $Z_{0}$ associated to which no local-in-time weak solution of \eqref{newty} exists.
\begin{rem}\label{wtf}
We would like to bring the reader's attention to the important fact that our definition \ref{physweak} of physical weak solution of Newton's equations in the absence of body forces does {\em not} permit contact rolling phenomena of the gas particles in free space. For instance, we do not accept as physical those spatial maps $X:\mathbb{R}\rightarrow\mathcal{P}_{2}(\mathsf{P}_{\ast})$ for which there exist real numbers $\tau_{-}<\tau_{+}$ with the property that
\begin{equation*}
\mathsf{P}(t)\cap\ov{\mathsf{P}}(t)=\left\{
\begin{array}{ll}
\varnothing & \quad \text{if}\hspace{2mm}t<\tau_{-}, \vspace{2mm}\\
\{P(t)\} & \quad \text{if}\hspace{2mm}\tau_{-}< t\leq\tau_{+}, \vspace{2mm}\\
\varnothing & \quad \text{if}\hspace{2mm}t>\tau_{+}, \vspace{2mm},
\end{array}
\right.
\end{equation*}
where $t\mapsto P(t)$ is of class $C^{2}$ on $(\tau_{-}, \tau_{+})$. Some authors\footnote{The author extends his sincere gratitude to Patrick Ballard for clarifying issues to him on this matter.} {\em do} allow for such behaviour, notably \textsc{Ballard} \cite{ballard2000dynamics}, whose general theory covers the case when the set $\mathsf{P}_{\ast}$ admits a real-analytic boundary curve. Our objection is straighforward: we object to such solutions as being physically admissible simply because they {\em do not satisfy} the ODE
\begin{equation}
M\ddot{X}(t)=0
\end{equation}
pointwise in the classical sense for each $t$ in the open interval $(\tau_{-}, \tau_{+})$, but rather the equation $M\ddot{X}(t)=Q(X(t), t)$ for some smooth non-trivial $Q:\mathbb{R}^{6}\times (\tau_{-}, \tau_{+})\rightarrow\mathbb{R}^{6}$. Alternatively, for those data leading to such contact rolling solutions in \cite{ballard2000dynamics}, we suggest it be preferable to aim for the construction of a distributional solution $X$ of \eqref{newty} that admits infinitely-many collisions on both the left and the right of $\tau_{-}\in\mathbb{R}$. Notably, such a spatial map would satisfy \eqref{newty} pointwise almost everywhere on $\mathbb{R}$. 
\end{rem}
\section{The Collision of Hard Disks: Characterisation of $G_{\beta}^{\Box}(\mathsf{D}_{\ast}; \mathcal{F})$}\label{harddeesk}
In this simple section, we tackle the much simpler case of hard disks before extending our approach to hard ellipses in section \ref{diskandell} below. The focus henceforth is on the characterisation of the sets of velocity germs $G_{\beta}^{\Box}(\mathsf{D}_{\ast}; \mathcal{X})$, when $\mathcal{X}=\mathcal{X}(\mathbb{R}, \mathbb{R}^{6})$ is chosen to be the vector space 
\begin{equation*}
\mathcal{F}(\mathbb{R}, \mathbb{R}^{6}):=\left\{
V\in L^{1}_{\mathrm{loc}}(\mathbb{R}, \mathbb{R}^{6})\,:\, V(t)=U\hspace{2mm}\text{for some}\hspace{2mm}U\in\mathbb{R}^{6}\hspace{2mm}\text{for a.e.}\hspace{2mm}t\in\mathbb{R}
\right\}.
\end{equation*}
This characterisation will aid us in the proof of our main theorem \ref{mainresult}. We employ the `tracking function' $\Phi:\mathbb{R}\times L^{1}_{\mathrm{loc}}(\mathbb{R}, \mathbb{R}^{6})\rightarrow\mathbb{R}$ defined by
\begin{equation*}
\Phi(t; W):=|x(t)-\ov{x}(t)|^{2}-1,
\end{equation*}
where $X=[x(t), \ov{x}(t), \vartheta(t), \ov{\vartheta}(t)]$ is defined in \eqref{com} and \eqref{ori} above. Clearly, $W\in \widetilde{\Sigma}_{\beta}^{-}(\mathsf{D}_{\ast})$ if and only if $\Phi(t; W)\leq 0$ for all $t\leq 0$. We also establish the following definition.
\begin{defn}
For any $\beta\in\mathbb{T}^{2}$, we define
\begin{align}
\Sigma_{\beta}^{-}(\mathsf{D}_{\ast}):=\left\{
U\in\mathbb{R}^{6}\,: \gamma_{\beta}\cdot U\leq 0
\right\}, \vspace{2mm} \notag\\
\Sigma_{\beta}^{+}(\mathsf{D}_{\ast}):=\left\{
U\in\mathbb{R}^{6}\,: \gamma_{\beta}\cdot U\geq 0
\right\}, \vspace{2mm} \notag \\
\Sigma_{\beta}^{0}(\mathsf{D}_{\ast}):=\left\{
U\in\mathbb{R}^{6}\,: \gamma_{\beta} \cdot U=0
\right\} \notag,
\end{align}
where $\gamma_{\beta}\in\mathbb{R}^{6}$ is the vector
\begin{equation*}
\gamma_{\beta}:=\left(
\begin{array}{c}
-2n_{\beta, 1} \\
-2n_{\beta, 2} \\
2n_{\beta, 1} \\
2n_{\beta, 2} \\
0 \\
0
\end{array}
\right).
\end{equation*}
\end{defn}
\begin{rem}
We note that for any $V\in\mathcal{F}(\mathbb{R}, \mathbb{R}^{6})$, it holds that $[V]_{s}^{-}=[V]_{s}^{+}=[V]_{s}$ for any $s\in\mathbb{R}$ (c.f. definitions \ref{lrgerms} and \ref{germs}). As such, we may denote any germ in $G_{\beta}^{\Box}(\mathsf{P}; \mathcal{F})$ generated by $V$ unambiguously by $[V]_{0}$.
\end{rem}
\begin{thm}\label{mainell}
Suppose $\mathcal{X}=\mathcal{F}(\mathbb{R}, \mathbb{R}^{6})$. For each $\beta\in\mathbb{T}^{2}$ and $\Box\in\{-, +, 0\}$, there exists a map $\psi_{\beta}:\mathbb{R}^{6}\rightarrow G_{\beta}(\mathsf{D}_{\ast}; \mathcal{X})$ with the property
\begin{equation*}
\psi_{\beta}^{\Box}(\Sigma_{\beta}^{\Box}(\mathsf{D}_{\ast}))= G_{\beta}^{\Box}(\mathsf{D}_{\ast}; \mathcal{F}),
\end{equation*}
where $\psi_{\beta}^{\Box}:=\psi_{\beta}|_{\Sigma_{\beta}^{\Box}(\mathsf{D}_{\ast}; \mathcal{F})}$. Moreover, $\psi_{\beta}^{\Box}|_{\Sigma_{\beta}^{\Box}(\mathsf{D}_{\ast})}$ is a bijection for $\Box\in\{-, +, 0\}$.
\end{thm}
\begin{proof}
As every vector in $\mathbb{R}^{6}$ gives rise to a constant map in $L^{1}_{\mathrm{loc}}(\mathbb{R}, \mathbb{R}^{6})$, we define $\psi_{\beta}:\mathbb{R}^{6}\rightarrow G_{\beta}(\mathsf{D}_{\ast}; \mathcal{F})$ to be the natural inclusion operator given by
\begin{equation*}
\psi_{\beta}(U):=[U]_{0}\quad\text{for any}\hspace{2mm}U\in\mathbb{R}^{6}.
\end{equation*}
Manifestly, by definition of the space of germs $G_{\beta}(\mathsf{D}_{\ast}; \mathcal{F})$, this map is a bijection. As suggested above, we consider the tracking function $\Phi:\mathbb{R}\times L^{1}_{\mathrm{loc}}(\mathbb{R}, \mathbb{R}^{6})\rightarrow\mathbb{R}$ defined in above. We lead with the following lemma.
\begin{lem}\label{help1}
For all $\beta\in\mathbb{T}^{2}$, one has that $[V]_{0}\in G_{\beta}^{-}(\mathsf{D}_{\ast}; \mathcal{F})$ if and only if for every $W\in [V]_{0}$, there exists $\delta=\delta(W)>0$ such that
\begin{equation*}
\Phi(t)\geq 0 \quad \text{for}\hspace{2mm}-\delta(W)<t\leq 0.
\end{equation*}
Similarly, $[V]_{0}\in G_{\beta}^{+}(\mathsf{D}_{\ast}; \mathcal{F})$ if and only if for every $W\in [V]_{0}$ there exists $\delta=\delta(W)>0$ such that
\begin{equation*}
\Phi(t)\geq 0 \quad \text{for}\hspace{2mm}0\leq t <\delta(W).
\end{equation*}
\end{lem}
\begin{proof}
This is immediate from the definition of the tracking function $\Phi$.
\end{proof}
As a result of this lemma, one need only investigate the sign behaviour of $\Phi(\cdot; W)$ in a neighbourhood 0 in order to decide to which set $G_{\beta}^{\Box}(\mathsf{D}; \mathcal{F})$ the germ $[W]_{0}$ belongs. 

Suppose $U=[v, \ov{v}, \omega, \ov{\omega}]\in\mathbb{R}^{6}$ is given and fixed. We note that for any $W\in [U]_{0}\in G_{\beta}(\mathsf{D}_{\ast}; \mathcal{F})$, there exists $\delta(W)>0$ such that
\begin{equation}\label{rest}
(X|_{(-\delta(W), 0]})(t)=\left[
\begin{array}{c}
tv \\
d_{\beta}e(\psi)+t\ov{v}\\
t\omega \\
\theta+t\ov{\omega}
\end{array}
\right] \quad \text{for}\hspace{2mm}-\delta(W)<t\leq 0.
\end{equation}
One can check that
\begin{equation}\label{tay1}
\Phi(t; W)=(\gamma_{\beta}\cdot U)t+(U\cdot AU)t^{2}
\end{equation}
for $-\delta(W)<t\leq 0$, where $\gamma_{\beta}\in\mathbb{R}^{6}$ is defined by above, and $A_{\beta}^{\varepsilon}\in\mathbb{R}^{6\times 6}$ is 
\begin{equation*}
A_{\beta}:=\left(
\begin{array}{cccccc}
1 & 0 & -1 & 0 & 0 & 0\\
0 & 1 & 0 & -1 & 0 & 0\\
-1 & 0 & 1 & 0 & 0 & 0 \\
0 & -1 & 0 & 1 & 0 & 0 \\
0 & 0 & 0 & 0 & 0 & 0\\
0 & 0 & 0 & 0 & 0 & 0\\
\end{array}
\right).
\end{equation*}
To understand how the choice of $U\in\mathbb{R}^{6}$ affects the sign behaviour of $\Phi(\cdot; W)$ in a neighbourhood of 0 for any $W\in[U]_{0}$, we split our considerations into two cases:
\subsection*{Case I: $\gamma_{\beta}\cdot U\neq 0$}
For any $W\in[U]_{0}$, one has from \eqref{tay1} that
\begin{equation*}
\lim_{t\rightarrow 0-}|\Phi(t; W)-(\gamma_{\beta}\cdot U) t|=0.
\end{equation*}
It follows that there exists $0<\delta_{1}(W)\leq\delta(W)$ such that $\Phi(t; W)\geq 0$ for all $-\delta_{1}(W)<t\leq 0$ if $\gamma_{\beta}\cdot U<0$. We conclude from lemma \ref{help1} that
\begin{equation*}
\psi_{\beta}\left(\left\{U\in\mathbb{R}^{6}\,:\,\gamma_{\beta}\cdot U<0\right\}\right)\subseteq G_{\beta}^{-}(\mathsf{D}_{\ast}; \mathcal{F}).
\end{equation*}
Similarly, one can show that
\begin{equation*}
\psi_{\beta}\left(\left\{U\in\mathbb{R}^{6}\,:\, \gamma_{\beta}\cdot U>0\right\}\right)\subseteq G_{\beta}^{+}(\mathsf{D}_{\ast}; \mathcal{F}).
\end{equation*}
\subsection*{Case II: $\gamma_{\beta}\cdot U=0$}
Observing \eqref{tay1}, one has in this case that
\begin{equation*}
\lim_{t\rightarrow 0-}\left|\Phi(t; W)-(U\cdot A_{\beta}U)t^{2}\right|=0,
\end{equation*}
for any $W\in[U]_{0}$. However, since $A_{\beta}$ is positive semi-definite on $\mathbb{R}^{6}$, it follows that there exists $\delta(W)>0$ such that $\Phi(t; W)\geq 0$ for all $-\delta(W)<t<\delta(W)$. As such, we infer that
\begin{equation}\label{do}
\psi_{\beta}\left(\left\{U\in\mathbb{R}^{6}\,:\,\gamma_{\beta}\cdot U=0\right\}\right)\subseteq G_{\beta}^{0}(\mathsf{D}_{\ast}; \mathcal{F}).
\end{equation}
By definition of $G_{\beta}^{0}(\mathsf{D}_{\ast}; \mathcal{F})$, it follows that
\begin{align*}
\psi_{\beta}\left(\left\{U\in\mathbb{R}^{6}\,:\, \gamma_{\beta}\cdot U\leq 0\right\}\right)\subseteq G_{\beta}^{-}(\mathsf{D}_{\ast}; \mathcal{F}), \vspace{2mm} \notag\\
\psi_{\beta}\left(\left\{U\in\mathbb{R}^{6}\,:\,\gamma _{\beta}\cdot U\geq 0\right\}\right)\subseteq G_{\beta}^{+}(\mathsf{D}_{\ast}; \mathcal{F}).
\end{align*}
As $\psi_{\beta}$ is a bijection, the reverse inclusions also hold true, whence follows the proof of the theorem.
\end{proof}
We conclude this section with a statement that when the two sets are congruent disks, all velocity vectors in $G_{\beta}(\mathsf{D}_{\ast}; \mathcal{F})$ are admissible (in the sense of {\bf (Q'2)} and, in fact, in the sense of {\bf (Q2)}).
\begin{cor}
Under the same assumptions of theorem \ref{mainell}, one has $G_{\beta}^{\times}(\mathsf{D}_{\ast}; \mathcal{F})=\varnothing$.
\end{cor}
As such, $G_{\beta}(\mathsf{D}_{\ast}; \mathcal{F})$ admits a partition into {\em only} pre-collisional and post-collisional velocities. As we shall find in the following section, this is not true for $G_{\beta}(\mathsf{P}_{\ast}; \mathcal{F})$ for every $\mathsf{P}_{\ast}\in\mathcal{C}$; in particular, when $\mathsf{P}_{\ast}$ is taken to be an ellipse, remarkably one can find velocities therein which are neither pre- nor post-collisional for any prescribed frictionless boundary conditions.
\section{The Collision of Hard Ellipses: Understanding $G_{\beta}^{\Box}(\mathsf{E}_{\ast}; \mathcal{F})$}\label{diskandell}
Let us now consider the dynamics of disks and ellipses in $\mathbb{R}^{2}$, with the aim of answering {\bf (Q'1)} and {\bf (Q'2)} in this special case. For a given $0<\varepsilon<1$, we write $\mathsf{E}_{\ast}^{\varepsilon}\subset\mathbb{R}^{2}$ to denote the ellipse given by
\begin{equation*}
\mathsf{E}_{\ast}^{\varepsilon}:=\left\{
y=(y_{1}, y_{2})\in\mathbb{R}^{2}\,:\, (\varepsilon y_{1})^{2}+y_{2}^{2}\leq 1
\right\}.
\end{equation*}
We consider the tracking function $\Psi:\mathbb{R}\times L^{1}_{\mathrm{loc}}(\mathbb{R}, \mathbb{R}^{6})\rightarrow\mathbb{R}$ defined by
\begin{equation*}
\Psi(t; W):=F(q(t), t) \quad \text{for}\hspace{2mm}W\in L^{1}_{\mathrm{loc}}(\mathbb{R}, \mathbb{R}^{6}),
\end{equation*}
where $F:\mathbb{R}^{2}\times\mathbb{R}\rightarrow\mathbb{R}$ is given by 
\begin{align}
F(y, t):=F_{0}(R(\vartheta(t))^{T}[y-x(t)]), \vspace{2mm} \notag\\
F_{0}(y):= y\cdot E_{\varepsilon}y-1, \vspace{2mm} \notag\\
E_{\varepsilon}:=\left(
\begin{array}{cc}
\varepsilon^{2} & 0 \\
0 & 1
\end{array}
\right),
\end{align}
and
\begin{equation*}
q(t):=\ov{x}(t)+R(\ov{\vartheta}(t))R(\theta)^{T}(p_{\beta}-d_{\beta}e(\psi)),
\end{equation*}
with $X(t):=[x(t), \ov{x}(t), \vartheta(t), \ov{\vartheta}(t)]$ defined in \eqref{com} and \eqref{ori} above, with $\vartheta\equiv 0$ and $\ov{\vartheta}\equiv\theta$. In the remainder of section \ref{diskandell}, for aesthetic reasons we suppress the dependence of the spatial maps $p_{\beta}, q_{\beta}, n_{\beta}$ and $d_{\beta}$ on the parameter $\varepsilon>0$. We require the following definition.
\begin{defn}
For any $0<\varepsilon<1$ and $\beta\in\mathbb{T}^{2}$, we define
\begin{align}
\Sigma_{\beta}^{-}(\mathsf{E}_{\ast}^{\varepsilon}):=\left\{
U\in\mathbb{R}^{6}\,: \nu_{\beta}^{\varepsilon}\cdot U\leq 0
\right\}, \vspace{2mm} \notag\\
\Sigma_{\beta}^{+}(\mathsf{E}_{\ast}^{\varepsilon}):=\left\{
U\in\mathbb{R}^{6}\,: \nu_{\beta}^{\varepsilon}\cdot U\geq 0
\right\}, \vspace{2mm} \notag \\
\Sigma_{\beta}^{0}(\mathsf{E}_{\ast}^{\varepsilon}):=\left\{
U\in\mathbb{R}^{6}\,: \nu_{\beta}^{\varepsilon}\cdot U=0
\right\},
\end{align}
where $\nu_{\beta}^{\varepsilon}\in\mathbb{R}^{6}$ is the vector
\begin{equation}\label{noo}
\nu_{\beta}^{\varepsilon}:=\left[
\begin{array}{c}
-n_{\beta} \\
n_{\beta} \\
-p_{\beta}^{\perp}\cdot n_{\beta} \\
q_{\beta}^{\perp}\cdot n_{\beta}
\end{array}
\right].
\end{equation}
\end{defn}
We now prove the following theorem.
\begin{thm}\label{mainell}
For each $0<\varepsilon<\frac{1}{2}$, there exists an open set $B_{\varepsilon}\in\mathbb{T}^{2}$ such that
\begin{equation*}
G_{\beta}^{\times}(\mathsf{E}_{\ast}^{\varepsilon}; \mathcal{F})\neq \varnothing.
\end{equation*}
for all $\beta\in B_{\varepsilon}$.
\end{thm}
\begin{proof}
Suppose $U=[v, \ov{v}, \omega, \ov{\omega}]\in\mathbb{R}^{6}$ is given and fixed. We note that for any $W\in [U]_{0}\in G_{\beta}(\mathsf{E}_{\ast}^{\varepsilon}; \mathcal{F})$, there exists $\delta(W)>0$ such that
\begin{equation}\label{rest}
(X|_{(-\delta(W), 0]})(t)=\left[
\begin{array}{c}
tv \\
d_{\beta}e(\psi)+t\ov{v}\\
t\omega \\
\theta+t\ov{\omega}
\end{array}
\right] \quad \text{for}\hspace{2mm}-\delta(W)<t\leq 0.
\end{equation}
As $\Psi(\cdot; W)$ is real-analytic in a neighbourhood of 0, one can show that
\begin{equation}\label{tay1}
\Psi(t; W)=\Psi'(0)t+\frac{\Psi''(0)}{2!}t^{2}+\mathcal{O}(t^{3})\quad \text{as}\hspace{2mm}t\rightarrow 0-,
\end{equation}
with the Taylor coefficients $\Psi'(0)$ and $\Psi''(0)$ being given by
\begin{equation*}
\Psi'(0):= 2 E_{\varepsilon}p_{\beta}\cdot\left(\ov{v}+\ov{\omega}q_{\beta}^{\perp}-v-\omega p_{\beta}^{\perp}\right)
\end{equation*}
and
\begin{equation*}
\Psi''(0):= U\cdot A_{\beta}^{\varepsilon}U,
\end{equation*}
where $A_{\beta}^{\varepsilon}\in\mathbb{R}^{6\times 6}$ is realised as
\begin{align}
U\cdot A_{\beta}^{\varepsilon}U:=4\omega (E_{\varepsilon}p_{\beta})^{\perp}\cdot (\ov{v}+\ov{\omega}q_{\beta}^{\perp}-v-\omega p_{\beta}^{\perp}) +2E_{\varepsilon}p_{\beta}\cdot (-\ov{\omega}^{2}q_{\beta}+\omega^{2}p_{\beta}) \notag \vspace{2mm}\\
+2(\ov{v}+\ov{\omega}q_{\beta}^{\perp}-v-\omega p_{\beta}^{\perp})\cdot E_{\varepsilon}(\ov{v}+\ov{\omega}q_{\beta}^{\perp}-v-\omega p_{\beta}^{\perp}).
\end{align}
To understand how the choice of $U\in\mathbb{R}^{6}$ affects the sign behaviour of $\Psi(\cdot; W)$ in a neighbourhood of 0 for any $W\in[U]_{0}$, as before we split our considerations into two cases:
\subsection*{Case I: $\nu_{\beta}^{\varepsilon}\cdot U\neq 0$}
For any $W\in[U]_{0}$, as one has from the Taylor expansion \eqref{tay1} that
\begin{equation*}
\lim_{t\rightarrow 0-}|\Psi(t; W)-(\nu_{\beta}^{\varepsilon}\cdot U) t|=0,
\end{equation*}
it follows that there exists $0<\delta_{1}(W)\leq\delta(W)$ such that $\Psi(t; W)\geq 0$ for all $-\delta_{1}(W)<t\leq 0$ if $\nu_{\beta}^{\varepsilon}\cdot U<0$. We conclude from lemma \ref{help1} that
\begin{equation*}
\psi_{\beta}^{\varepsilon}\left(\mathrm{int}\,\Sigma_{\beta}^{-}(\mathsf{E}_{\ast}^{\varepsilon})\right)\subseteq G_{\beta}^{-}(\mathsf{E}_{\ast}^{\varepsilon}; \mathcal{F}).
\end{equation*}
In the very same manner, one has that
\begin{equation*}
\psi_{\beta}^{\varepsilon}\left(\mathrm{int}\,\Sigma_{\beta}^{+}(\mathsf{E}_{\ast}^{\varepsilon})\right)\subseteq G_{\beta}^{+}(\mathsf{E}_{\ast}^{\varepsilon}; \mathcal{F}).
\end{equation*}
\subsection*{Case II: $\nu_{\beta}^{\varepsilon}\cdot U=0$}
It is our aim to show that the sign of the second Taylor coefficient is not guaranteed to be non-negative. By vanishing of the first Taylor coefficient in \eqref{tay1}, one has
\begin{equation}\label{thegoodexp}
\lim_{t\rightarrow 0}\left|\Psi(t; W)-U\cdot A_{\beta}^{\varepsilon}U \frac{t^{2}}{2!}\right|=0,
\end{equation}
for any $W\in[U]_{0}$. Unlike in the case of disks, we claim that there exists $\beta_{\ast}(\varepsilon)\in\mathbb{T}^{2}$ (and so an open neighbourhood thereof) such that the quadratic form $U\cdot A_{\beta_{\ast}(\varepsilon)}^{\varepsilon}U$ can assume negative values on $\mathbb{R}^{6}$. In particular, we have the following lemma.
\begin{lem}\label{heat}
For every $0<\varepsilon<\frac{1}{2}$, there exists $\beta_{\ast}(\varepsilon)\in\mathbb{T}^{2}$, a vector $U_{\ast}(\varepsilon)\in\Sigma_{\beta}^{0}(\mathsf{E}_{\ast}^{\varepsilon})$ and a radius $\rho_{\ast}(\varepsilon)>0$ such that
\begin{equation*}
U\cdot A_{\beta_{\ast}(\varepsilon)}^{\varepsilon}U<0 \quad \text{for all}\hspace{2mm}V\in B(U_{\ast}, \rho_{\ast})\cap \Sigma_{\beta}^{0}(\mathsf{E}_{\ast}^{\varepsilon}),
\end{equation*}
where $B(U, \rho)\subset\mathbb{R}^{6}$ is the open ball centred at $U$ with radius $\rho$.
\end{lem}
\begin{proof}
We focus our attention on those vectors $U\in\mathbb{R}^{6}$ of the form
\begin{equation}\label{spesh}
U=\left(
\begin{array}{c}
0 \\
0 \\
\ov{v}_{1} \\
\ov{v}_{2} \\
1 \\
0
\end{array}
\right),
\end{equation}
for some $\ov{v}\in\mathbb{R}^{2}\setminus\{0\}$, namely those velocities for which one ellipse is rotating but not translating, while the other is translating but not rotating. For any such vector \eqref{spesh} satisfying the additional criterion that $U\in\Sigma_{\beta}^{0}(\mathsf{E}_{\ast}^{\varepsilon})$, one can show that
\begin{equation*}
\Psi''(0)=\frac{2\varepsilon^{2}}{p_{\beta, 2}^{2}}\left[ \ov{v}_{1}+\left(\frac{\varepsilon^{2}-1}{\varepsilon^{2}}\right)p_{\beta, 2}^{3}\right]^{2}-2\underbrace{\left[\frac{(\varepsilon^{2}-1)^{2}}{\varepsilon^{2}}p_{\beta, 2}^{4}+\varepsilon^{2}+1-2\varepsilon^{2}|p_{\beta}|^{2}\right]}_{K_{\beta}:=},
\end{equation*}
assuming that $\beta$ is taken such that $p_{\beta, 2}\neq 0$. Note that $K_{\beta}$ is velocity-independent. If one can find $\beta\in\mathbb{T}^{2}$ such that $K_{\beta}>0$, it will be possible to find $\ov{v}_{1}=\ov{v}_{1}(\beta)\in\mathbb{R}$ such that $\Psi''(0)<0$, and so the claim of the lemma will follow by continuity of the map $U\mapsto U\cdot A_{\beta}^{\varepsilon}U$. 

In this direction, we note that for a given $\beta\in\mathbb{T}^{2}$ one can find $u_{\beta}\in\mathbb{S}^{1}$ such that
\begin{equation*}
p_{\beta}=\left(
\begin{array}{c}
\varepsilon^{-1}\cos u_{\beta}\\
\sin u_{\beta}
\end{array}
\right).
\end{equation*}
The map $\beta\mapsto u_{\beta}$ is surjective. Writing $p_{\beta}$ in this way, the constant $K_{\beta}$ admits the representation
\begin{equation}\label{weee}
K_{\beta}=\frac{(\varepsilon^{2}-1)^{2}}{\varepsilon^{2}}\left[ \sin^{2}u_{\beta}+\frac{\varepsilon^{2}}{1-\varepsilon^{2}}\right]^{2}-1.
\end{equation}
Naturally, by considering the polynomial $P:[0, 1]\rightarrow\mathbb{R}$ associated to \eqref{weee} defined by
\begin{equation*}
P(x):=\frac{(\varepsilon^{2}-1)^{2}}{\varepsilon^{2}}\left[ x+\frac{\varepsilon^{2}}{1-\varepsilon^{2}}\right]^{2}-1,
\end{equation*}
one can show $P(x)>0$ for $\frac{\varepsilon}{1+\varepsilon}<x\leq 1$. By surjectivity of the map $\beta\mapsto u_{\beta}$, it follows there exists $\beta_{\ast}(\varepsilon)\in\mathbb{T}^{2}$ such that $K_{\beta_{\ast}(\varepsilon)}>0$. The proof of the lemma follows.
\end{proof}
From lemma \ref{heat}, from \eqref{thegoodexp} we have that if one takes $U\in B(U_{\beta})$, then for any $W\in [U]_{0}$ there exists $\delta(W)>0$ such that
\begin{equation*}
\Psi(t; W)<0 \quad \text{for all}\hspace{2mm}-\delta(W)<t<\delta(W).
\end{equation*}
As a result, the set of all inadmissible velocities $G_{\beta_{\ast}(\varepsilon)}^{\times}(\mathsf{E}_{\ast}^{\varepsilon}; \mathcal{F})$ corresponding to the spatial configuration $\beta_{\ast}(\varepsilon)\in \mathbb{T}^{2}$ is non-empty, whence $\mathcal{G}^{\times}(\mathsf{E}_{\ast}^{\varepsilon}; \mathcal{F})\neq \{\varnothing\}$. This concludes the proof of the theorem.
\end{proof}
As an immediate corollary, we have the following.
\begin{cor}
Let $\{\sigma_{\beta}\}_{\beta\in\mathbb{T}^{2}}$ be a family of frictionless and physical scattering maps. One can find $Z_{0}\in\mathcal{P}_{2}(\mathsf{E}_{\ast}^{\varepsilon})\times\mathbb{R}^{6}$ for which there exists no associated local-in-time physical weak solution of \eqref{newty}.
\end{cor}
\section{Closing Remarks}\label{closrem}
In the case of no external forcing $(F\equiv 0)$, theorem \ref{mainresult} makes it clear that one cannot establish an existence theory for {\bf (IBVP)} solely in the class of dynamics for which $t\mapsto X(t)$ is piecewise linear and both left- and right-differentiable on $\mathbb{R}$. One must therefore attempt to construct solutions associated to such initial data -- in the sense of definition \ref{formaldeffy} -- which lie in a class of maps of {\em lower regularity}. Notably, one cannot expect $t\mapsto V(t)$ to be left-differentiable on $\mathbb{R}$. In particular, this necessary drop in regularity will give rise to solutions of {\bf (IBVP)} for which infinitely-many collisions of the sets $\mathsf{P}$ and $\ov{\mathsf{P}}$ occur on a compact subinterval of $\mathbb{R}$. 

\subsection*{Acknowledgements}
The author would like to thank Gilles Francfort, Laure Saint-Raymond and Jean Taylor for stimulating discussions on the dynamics of rigid bodies when he was a Courant Instructor at the Courant Institute of Mathematical Sciences, New York City. He would also like to thank Xiaoyu Zheng, Peter-Palffy-Muhoray, Epifanio Virga, Robin Knops and Heiko Gimperlein for their insightful comments.

\subsection*{Declarations of Interest}
None.

\appendix
%

\bibliography{biblio}

\end{document}